\documentclass[12pt]{article}
\usepackage{amsmath,amsthm,amsfonts,amssymb, color,xcolor,subcaption,graphicx} %,leftindex}
\usepackage[normalem]{ulem}
\newtheorem{theorem}{Theorem}[section]

\newtheorem{lemma}[theorem]{Lemma}

\def\cB{\mathcal{B}}

\def\cF{\mathcal{F}}
\def\cG{\mathcal{G}}
\def\cH{\mathcal{H}}

\def\cZ{\mathcal{Z}}

\def\bE{\mathbb{E}}

\def\bR{\mathbb{R}}

\newcommand{\fm}{\mathfrak{m}}
\begin{document}
\title{Central Limit Theorem for Stochastic Nonlinear Heat Equation with Pure-Jump Lévy White Noise }

\author{ Matis Le Gall\footnote{DER de mathématiques, ENS Paris-Saclay, Gif-sur-Yvette, France. E-mail address: matis.le\_gall@ens-paris-saclay.fr.  } \and   Jinxin Wang\footnote{Corresponding author. University of Ottawa, Department of Mathematics and Statistics, 150 Louis Pasteur Private, Ottawa, Ontario, K1N 6N5, Canada. E-mail address: jwang023@uottawa.ca.}}

\date{September 24, 2026}
\maketitle

\begin{abstract}
\noindent
In this article, we consider the stochastic nonlinear heat equation driven by
L\'evy space-time white noise in dimension one. For the spatial average of the
solution, we prove quantitative and functional central limit theorems under
 $m_1+m_{2p}<\infty$ for some $p\in(1,\frac32)$. These results extend the Gaussian fluctuation theory for the
parabolic Anderson model to the nonlinear setting. The main new feature is a
minimum-type term in the second Malliavin derivative estimate caused by the
nonlinear coefficient.
\end{abstract}

\noindent {\em MSC 2020:} Primary 60H15; Secondary 60F05, 60G60, 60G51
%60H15=SPDEs
%60G60=random fields
%60G51: Processes with independent increments; L\'evy processes
%60F05 Central limit and other weak theorems

\vspace{1mm}

\noindent {\em Keywords:} stochastic partial differential equations, random fields, Malliavin calculus, Poisson random measure, L\'evy noise

%\pagebreak
%\tableofcontents

%\include{Preliminaries}
%\include{Theorem1}
%\include{H}
\section{Introduction}
The study of Gaussian fluctuations for spatial averages of stochastic heat
equations was initiated, in the setting of Gaussian space-time white noise, by
Huang, Nualart and Viitasaari \cite{HNV20}. They proved
a {\em quantitative central limit theorem (QCLT)} and a {\em functional central limit theorem (FCLT)} for the
normalized centered spatial integral of the solution
over $[-R,R]$ as $R\to\infty$,
\begin{equation}
\label{FR-old}
 F_R(t)
:=
\frac{1}{\sigma_R}
\left(
\int_{-R}^{R} u(t,x)\,dx
-
\bE\int_{-R}^{R} u(t,x)\,dx
\right),
\end{equation}
where $R>0$,  $u(t,x)$ is the solution of the stochastic heat equation driven by Gaussian space-time white noise considered in \cite{HNV20}, and $\sigma_R^2
=
\operatorname{Var}
\left(
\int_{-R}^{R} u(t,x)\,dx
\right)$. Subsequent studies extended these results to spatially colored Gaussian
noise~\cite{HNVZ20}, delta initial
data~\cite{KNP21}, temporally colored parabolic Anderson
models~\cite{NXZ22}, and systems of stochastic heat
equations~\cite{NS24}.

\medskip

We develop the corresponding fluctuation theory for a nonlinear stochastic heat equation driven by L\'evy space-time white noise. Specifically, we consider the equation with
constant initial condition
\begin{equation}
\label{nonlinear}
\begin{cases}
\displaystyle
\frac{\partial u}{\partial t}(t,x)
=
\frac{1}{2}\frac{\partial^2 u}{\partial x^2}(t,x)
+\sigma(u(t,x))\dot L(t,x),
\quad t>0,\ x\in\bR,\\[0.3em]
u(0,x)=1,\quad x\in\bR,
\end{cases}
\end{equation}
where $\dot L$ denotes the {\em L\'evy white noise} on
$\bR_+\times\bR$, and $\sigma:\bR\to\bR$ is a deterministic globally
Lipschitz function.

\medskip

 Let $N$ be a Poisson random measure on
$\bR_+\times\bR\times\bR_0$, with intensity $dtdx\nu(dz)$, and let
$\widehat N$ be its compensated version, where $\bR_0=\bR\setminus\{0\}$.
The L\'evy white noise $L$ is defined by
\begin{equation}
\label{def-L}
L(\varphi)
=
\int_{\bR_+\times\bR\times\bR_0}
\varphi(t,x)z\,
\widehat N(dt,dx,dz),
\qquad
\varphi\in L^2(\bR_+\times\bR).
\end{equation}
 For $p\ge1$, set $m_p:=\int_{\bR_0}|z|^p\nu(dz)$. Throughout, we
assume that $0<m_2<\infty$.

\medskip

\medskip

A predictable process $\{u(t,x):t\ge0,x\in\bR\}$ is called a mild solution of
\eqref{nonlinear} if
\begin{equation}
\label{solution}
u(t,x)
=
1+
\int_0^t\int_{\bR}
G_{t-s}(x-y)\sigma(u(s,y))L(ds,dy),
\end{equation}
where
\[
G_t(x)
=
\frac{1}{\sqrt{2\pi t}}
\exp\left(-\frac{x^2}{2t}\right),
\qquad t>0,\ x\in\bR,
\]
is the heat kernel. Existence for nonlinear stochastic heat equations driven by general
L\'evy noise was established in \cite{C17}, while path properties and
intermittency were studied in \cite{CDH19} and \cite{CK19}, respectively.
 These studies allow for more general L\'evy noises, for which the
finite-variance condition $m_2<\infty$ need not hold. For the finite-variance L\'evy white noise considered here, existence
and uniqueness of the random-field solution follow from
\cite{BN17}.
Throughout the paper we assume that $\sigma(1)\neq0$.
Otherwise, $u(t,x)\equiv1$ is the unique solution of
\eqref{nonlinear}, and the fluctuation problem is trivial. Since the stochastic integral in \eqref{solution} is centered,
$\bE[u(t,x)]=1$.
In the present paper,  our object of study is the centered spatial average
\begin{equation}
\label{FR-new}
F_R(t)
:=
\int_{-R}^{R}(u(t,x)-1)\,dx,
\qquad R>0.
\end{equation}

\medskip

When $\sigma(u)=u$, equation \eqref{nonlinear} becomes the
{\em parabolic Anderson model} (pAm) driven by L\'evy white noise. Spatial ergodicity, the limiting
covariance, a QCLT, and an FCLT for \eqref{FR-new} were established for that
model in \cite{BLW26}. The present work shows that this Gaussian fluctuation
picture persists for a globally Lipschitz nonlinear coefficient.

\medskip

A standard way to prove Gaussian fluctuations for spatial averages of SPDEs is
the Malliavin--Stein approach.  The key input  is a
$p$-moment estimate for the Malliavin derivatives of the solution.
Compared with the pAm analysis in \cite{BLW26}, the nonlinear coefficient
$\sigma$ does not change the structure of the first Malliavin derivative
estimate. More precisely, we prove that for any $p\in[2,3)$ such that
$m_p<\infty$,
\begin{equation}
\label{Du}
\big\|D_{r,y,z}u(t,x)\big\|_p
\le
C_{T,p}|z|\,g_{t-r}^{(p)}(x-y).
\end{equation}
 where $g_s^{(p)}(x)
:=
G_s(x)+G_s(x)^{2/p}$. The restriction $p<3$ comes from the integrability of the heat
kernel. For the stochastic nonlinear wave equation driven by
L\'evy white noise, Balan and Zheng \cite{BZ26} use a Poisson
Malliavin--Stein bound involving only first Malliavin derivatives.
Their proof relies on estimates for the fourth moment of the
first Malliavin derivative, so their argument cannot be applied
directly in our setting.

\medskip

We therefore use a second-order Poincar\'e inequality on the Poisson
space, which requires an estimate for $D^2u(t,x)$. This is where the nonlinear
coefficient creates a new minimum-type term. For $r_1<r_2$, the pAm
contribution has the product form
\[
|z_1z_2|\,
g_{t-r_2}^{(p)}(x-y_2)
g_{r_2-r_1}^{(p)}(y_2-y_1),
\]
whereas the nonlinear coefficient produces the additional term
\begin{equation}
\label{min-term}
\min\left\{
|z_1|g_{t-r_1}^{(p)}(x-y_1),
|z_2|g_{t-r_2}^{(p)}(x-y_2)
\right\}.
\end{equation}
The product term is already present in \cite{BLW26}. The minimum term
has no analogue in the pAm, and its control is the main technical contribution
of this paper.

\medskip

 For $T>0$, let $D[0,T]$ denote the space of c\`adl\`ag functions on
$[0,T]$. We write $J_1$ for the Skorokhod topology and $U$ for the topology of
uniform convergence. Our main results are as follows.

\begin{theorem}[Quantitative CLT]
\label{QCLT}
Assume that there exists $p\in(1,\frac32)$ such that
\begin{equation}
\label{mp-m2p}
m_1+m_{2p}<\infty.
\end{equation}
Then, for every $t>0$,
\[
{\rm dist}\left(\frac{F_R(t)}{\sigma_R(t)},Z\right)
\le C_t R^{-\left(1-\frac1p\right)} \qquad \text{for } R\ge1.
\]
Here $C_t>0$ depends on $t$, $Z$ is a standard normal random variable,
$\sigma_R^2(t)=\operatorname{Var}(F_R(t))$, and ${\rm dist}$ denotes any one
of the distances $d_W$, $d_{FM}$, and $d_K$, namely the $1$-Wasserstein,
Fortet--Mourier, and Kolmogorov distances, respectively.
\end{theorem}

\begin{theorem}[Functional CLT]
\label{FCLT}
Under the hypotheses of Theorem~\ref{QCLT}, for every $R>0$, the process
$\{F_R(t)\}_{t\geq0}$ has a c\`adl\`ag modification (denoted also $F_R$).
Moreover, for any $T>0$,
\[
\frac{1}{\sqrt R}F_R(\cdot)\xrightarrow{d}\cG(\cdot)
\quad \mbox{in $(D[0,T],J_1)$ as $R\to\infty$},
\]
where $\{\cG(t)\}_{t\geq0}$ is a centered continuous Gaussian process with covariance
\[
\bE[\cG(t)\cG(s)]
=\Sigma_{t,s}
:=2m_2\int_0^{t\wedge s}\bE|\sigma(u(r,0))|^2\,dr.
\]
The convergence also holds in $(D[0,T],U)$.
\end{theorem}

Section~2 establishes the required moment bounds for $u$, $Du$, and
$D^2u$. Section~3 applies these bounds to the spatial averages and proves
Theorems~\ref{QCLT} and \ref{FCLT}.

\section{Key estimates on Malliavin derivatives}
\subsection{Poisson Malliavin derivative and chain estimates}
Let $({\bf Z},\cZ,\fm)$ be the measure space
\[
({\bf Z},\cZ,\fm)
=\big(\bR_+\times\bR\times\bR_0,
\cB(\bR_+)\otimes\cB(\bR)\otimes\cB(\bR_0),
{\rm Leb}\otimes{\rm Leb}\otimes\nu\big),
\]
 and set $\cH=L^2({\bf Z},\cZ,\fm)$. For $q\ge1$, we write
$\|X\|_q=(\bE|X|^q)^{1/q}$. Constants denoted by $C$ or $C_{T,p}$ may
change from line to line. For Malliavin calculus with respect to the
compensated Poisson random measure $\widehat N$, we refer to \cite{NN18}.

\medskip

We recall the add-one Malliavin derivative. Let ${\bf N}_{\bf Z}$ be the space of all $\sigma$-finite counting measures
on $({\bf Z},\cZ)$. If $F$ is a real-valued $\cF^N$-measurable random
variable, then there exists a measurable function
$f:{\bf N}_{\bf Z}\to\bR$ such that $F=f(N)$. For $\xi\in{\bf Z}$, the
add-one operator is defined by
\begin{equation}
\label{add-one}
D_\xi^+F
:=
f(N+\delta_\xi)-f(N),
\end{equation}
whenever the right-hand side is well defined. 

It is known that the add-one operator $D^+$ coincides with the Malliavin
derivative operator $D$ on ${\rm dom}(D)$. Hence, we simply
write $D_\xi F$ instead of $D_\xi^+F$. For $\xi_1,\xi_2\in{\bf Z}$, the
second-order difference is given by
\begin{equation}
\label{second-add-one}
D^2_{\xi_1,\xi_2}F
=
f(N+\delta_{\xi_1}+\delta_{\xi_2})
-
f(N+\delta_{\xi_1})
-
f(N+\delta_{\xi_2})
+
f(N),
\end{equation}
whenever the right-hand side is well defined.

The following elementary chain estimates will be used repeatedly.

\begin{lemma}
\label{lem:chain-estimates}
Let $\phi:\bR\to\bR$ be a Lipschitz function with Lipschitz constant
${\rm Lip}(\phi)$. Then, whenever the following differences are well defined,
\begin{equation}
\label{D-chain}
|D_\xi\phi(F)|
\le
{\rm Lip}(\phi)|D_\xi F|,
\end{equation}
and
\begin{equation}
\label{D2-chain}
|D^2_{\xi_1,\xi_2}\phi(F)|
\le
{\rm Lip}(\phi)|D^2_{\xi_1,\xi_2}F|
+
2{\rm Lip}(\phi)
\min\big\{
|D_{\xi_1}F|,
|D_{\xi_2}F|
\big\}.
\end{equation}
\end{lemma}

\begin{proof}
The first estimate follows directly from the Lipschitz property and
\eqref{add-one}. For the second estimate, we have
\begin{align*}
D^2_{\xi_1,\xi_2}\phi(F)
&=
\phi\big(f(N+\delta_{\xi_1}+\delta_{\xi_2})\big)
-
\phi\big(f(N+\delta_{\xi_1})\big)
-
\phi\big(f(N+\delta_{\xi_2})\big)
+
\phi(f(N)) \\
&=
\phi\left(
F+D_{\xi_1}F+D_{\xi_2}F+D^2_{\xi_1,\xi_2}F
\right)
-
\phi\left(F+D_{\xi_1}F\right)
-
\phi\left(F+D_{\xi_2}F\right)
+
\phi(F).
\end{align*}
We use the elementary Lipschitz estimate
\[
|\phi(x+a+b+c)-\phi(x+a)-\phi(x+b)+\phi(x)|
\le
{\rm Lip}(\phi)
\big(
|c|+2\min\{|a|,|b|\}
\big).
\]
Applying this estimate with $x=F$,  $a=D_{\xi_1}F$, $b=D_{\xi_2}F$ and $c=D^2_{\xi_1,\xi_2}F$ gives the desired bound.

\end{proof}

\subsection{Malliavin derivative estimates}
 We first show that the solution has finite $p$-th moments for
$p\in[2,3)$ with $m_p<\infty$, and then estimate its first and second
Malliavin derivatives. The arguments extend the corresponding pAm estimates,
so we emphasize the modifications caused by the nonlinear coefficient. We write $L_\sigma$ for a Lipschitz constant of $\sigma$. 
Throughout this section, for $p\ge2$, set
\[
H_t^{(p)}(x)
=
G_t^2(x)+G_t^p(x) \qquad \text{and}\qquad g_t^{(p)}(x)
=
G_t(x)+G_t^{2/p}(x).
\]

We begin with the moment estimate.
\begin{theorem}
\label{thm-u-pth-moment-nonlinear}
Let $p\in [2,3)$ be such that $m_p<\infty$. Let $u$ be the mild solution of
\eqref{nonlinear}. Then, for any $T>0$,
\begin{equation}
\label{u-pth-moment}
    \sup_{(t,x)\in [0,T]\times \bR} \bE|u(t,x)|^p <\infty .
\end{equation}
\end{theorem}

\begin{proof}
Let $(u_n)_{n\geq 0}$ be the sequence of Picard iterations, defined by
$u_0(t,x)=1$ and
\begin{equation}
\label{Picard}
u_{n+1}(t,x)=1+\int_0^t \int_{\bR} G_{t-s}(x-y) \sigma (u_n(s,y)) L(ds,dy),
\quad n\geq 0.
\end{equation}
By \cite[Lemma~2.5]{BLW26}, the linear growth of $\sigma$, and
$\int_0^T\int_{\bR}H_s^{(p)}(y)\,dy\,ds<\infty$,
\[
\bE|u_{n+1}(t,x)|^p
\le
C_{T,p}
+C_{T,p}
\int_0^t\int_{\bR}
H_{t-s}^{(p)}(x-y)\bE|u_n(s,y)|^p\,dy\,ds .
\]
Iterating this renewal inequality as in \cite[Theorem~2.7]{BLW26}
yields a bound uniform in $n$. Passing to the limit in the Picard sequence
proves \eqref{u-pth-moment}.
\end{proof}

\begin{theorem}
\label{key-th-Du}
\begin{itemize}
\item[(i)] For every $t\geq0$ and $x\in\bR$, $u(t,x)\in{\rm dom}(D)$.
\item[(ii)] Let $p\in[2,3)$ satisfy $m_p<\infty$. For any
$0\le r<t\le T$, $x,y\in\bR$, and $z\in\bR_0$,
\begin{equation}
\label{Du-pth-moment}
 \| D_{r,y,z}u(t,x)\|_p\le C_{T,p}|z|g_{t-r}^{(p)}(x-y),
\end{equation}
where $C_{T,p}>0$ is non-decreasing in $T$.
\end{itemize}
\end{theorem}

\begin{proof}
By induction, $u_n(t,x)\in{\rm dom}(D)$ for every $n\ge0$.
Let $\xi=(r,y,z)$. From the recurrence relation \eqref{Picard} and the Heisenberg commutation principle, it follows that
\begin{equation*}
D_\xi u_{n+1}(t,x)=  zG_{t-r}(x-y)\sigma(u_n(r,y)) + \int_r^t\int_{\bR}
G_{t-s_1}(x-y_1)D_\xi\sigma(u_n(s_1,y_1))L(ds_1,dy_1).
\end{equation*}
 By \cite[Lemma~2.5]{BLW26}, \eqref{u-pth-moment}, and
\eqref{D-chain},
\begin{align*}
 \bE\left|D_{\xi}u_{n+1}(t,x)\right|^p \le &  2^{p-1}C_{T,p}|z|^pG_{t-r}^p(x-y)\\
 &+2^{p-1}A_{T,p}L_\sigma^p
\int_r^t\int_{\bR}
H_{t-s_1}^{(p)}(x-y_1)
\bE\left|D_{\xi}u_n(s_1,y_1)\right|^p
dy_1ds_1.
\end{align*}
The renewal argument in \cite[Theorem~3.1]{BLW26} now gives a bound
uniform in $n$. The closability of $D$ and convergence of the Picard sequence
yield (i) and \eqref{Du-pth-moment}.
\end{proof}

\medskip

For $k\ge1$ and $0\le r<t$, define
$T_k(r,t)=\{(t_1,\ldots,t_k):r<t_1<\cdots<t_k<t\}$. We use the
conventions $t_{k+1}=t$, $x_{k+1}=x$,
$\pmb{t}=(t_1,\ldots,t_k)$, and $\pmb{x}=(x_1,\ldots,x_k)$.
In \cite{BLW26},  it is proved that if $p\in[2,3)$ and
$T>0$, then for any $0\le r<t\le T$ and $x,y\in\bR$,
\begin{align}
\label{Gp-bound}
\sum_{k=1}^{\infty}
C_{T,p}^k
\int_{T_k(r,t)}
\int_{\bR^k}
\prod_{i=1}^{k}
H_{t_{i+1}-t_i}^{(p)}(x_{i+1}-x_i)
G_{t_1-r}^p(x_1-y)
\,d\pmb{x}\,d\pmb{t}
\le
C_{T,p}H_{t-r}^{(p)}(x-y).
\end{align}
The proof of \eqref{Gp-bound} splits into two cases:
the case in which all factors are of type $G^p$, and the
case in which at least one factor is of type $G^2$.
The following estimate is an extension of
\eqref{Gp-bound}. The difference is that the initial kernel
$G_{t_1-r}^p$ is replaced by $H_{t_1-r}^{(p)}$.

\begin{lemma}
Let $p\in[2,3)$ and $T>0$. Then, for any $0\le r<t\le T$ and
$x,y\in\bR$,
\begin{align}
\label{GpG2-bound}
\sum_{k=1}^{\infty}
C_{T,p}^k
\int_{T_k(r,t)}
\int_{\bR^k}
\prod_{i=1}^{k}
H_{t_{i+1}-t_i}^{(p)}(x_{i+1}-x_i)
H_{t_1-r}^{(p)}(x_1-y)
\,d\pmb{x}\,d\pmb{t}
\le
C_{T,p}
H_{t-r}^{(p)}(x-y).
\end{align}
\end{lemma}

\begin{proof}
 Since $H_{t_1-r}^{(p)}=G_{t_1-r}^p+G_{t_1-r}^2$, the first
contribution is bounded by \eqref{Gp-bound}. For the second, expand the product
of the $H^{(p)}$ kernels. Every resulting term contains a $G^2$ factor, so the
second case in the proof of \eqref{Gp-bound} applies.
\end{proof}

\begin{theorem}
\label{key-th-D2u}
If $m_1<\infty$, then
\begin{itemize}
\item[(i)] for every $t>0$ and $x\in\bR$, $u(t,x)\in{\rm dom}(D^2)$;
\item[(ii)] if $p\in[2,3)$ and $m_p<\infty$, then for any
$0\le t\le T$, $x\in\bR$, and
$\xi_i=(r_i,y_i,z_i)\in{\bf Z}$ with $r_i\in[0,t]$, $i=1,2$,
\begin{equation}
\label{D2u-nonlinear-bound}
\begin{aligned}
\|D^2_{\xi_1,\xi_2}u(t,x)\|_p
&\le
C_{T,p}|z_1z_2|
\begin{cases}
g_{t-r_2}^{(p)}(x-y_2)
g_{r_2-r_1}^{(p)}(y_2-y_1),
& \text{if } r_1<r_2, \\[0.3em]
g_{t-r_1}^{(p)}(x-y_1)
g_{r_1-r_2}^{(p)}(y_1-y_2),
& \text{if } r_2<r_1,
\end{cases}
\\
&\quad+
C_{T,p}
\min\left\{
|z_1|g_{t-r_1}^{(p)}(x-y_1),
|z_2|g_{t-r_2}^{(p)}(x-y_2)
\right\}.
\end{aligned}
\end{equation}
\end{itemize}
\end{theorem}

\begin{proof}
We prove the estimate for $0\le r_1<r_2<t\le T$. The case
$r_2<r_1$ follows by symmetry, and the diagonal $r_1=r_2$ is
$\fm^{\otimes2}$-null. By induction, $u_n(t,x)\in{\rm dom}(D^2)$ for every
$n\ge0$.
For the Picard approximations, applying $D_{\xi_1}$ to the equation for
$D_{\xi_2}u_{n+1}(t,x)$ gives
\begin{align}
\label{D2u-picard-nonlinear}
D^2_{\xi_1,\xi_2}u_{n+1}(t,x)
=&
z_2G_{t-r_2}(x-y_2)
D_{\xi_1}\sigma(u_n(r_2,y_2))
\nonumber\\
&+
\int_{r_2}^t\int_{\bR}
G_{t-s}(x-y)
D^2_{\xi_1,\xi_2}\sigma(u_n(s,y))
L(ds,dy).
\end{align}
Applying Rosenthal's inequality with
$\Phi(s,y)=D^2_{\xi_1,\xi_2}\sigma(u_n(s,y))$, we obtain
\begin{equation}
    \label{D2u-0}
\begin{aligned}
 \bE\left|D^2_{\xi_1,\xi_2}u_{n+1}(t,x)\right|^p
&\le
2^{p-1}
\bigg\{
|z_2|^pG_{t-r_2}^p(x-y_2)
\bE\left|D_{\xi_1}\sigma\big(u_n(r_2,y_2)\big)\right|^p\\
&\quad+
A_{T,p}
\int_{r_2}^t\int_{\bR}
H_{t-s}^{(p)}(x-y)
\bE\left|
D^2_{\xi_1,\xi_2}\sigma(u_n(s,y))
\right|^p
dy ds
\bigg\}.
\end{aligned}
\end{equation}
By \eqref{D-chain} and 
 \eqref{Du-pth-moment},
\begin{equation}
    \label{D2u-1}
\bE\left|D_{\xi_1}\sigma\big(u_n(r_2,y_2)\big)\right|^p\le
L_\sigma^p C_{T,p}^p
|z_1|^p
H_{r_2-r_1}^{(p)}(y_2-y_1). 
\end{equation}
By \eqref{D2-chain},
\begin{equation}
\label{D2u-2}
\begin{aligned}
\bE\left|D^2_{\xi_1,\xi_2}\sigma\big(u_n(s,y)\big)\right|^p
&\le
C_pL_\sigma^p
\bE\left|D^2_{\xi_1,\xi_2}u_n(s,y)\right|^p
+
C_pL_\sigma^p
\bE\left[
\min_{i=1,2}
\left|D_{\xi_i}u_n(s,y)\right|^p
\right]  \\
&\le
C_pL_\sigma^p
\bE\left|D^2_{\xi_1,\xi_2}u_n(s,y)\right|^p
+
C_{T,p}L_\sigma^p
\min_{i=1,2}
\left\{
|z_i|^pH_{s-r_i}^{(p)}(y-y_i)
\right\}.
\end{aligned}
\end{equation}
Combining \eqref{D2u-0} with \eqref{D2u-1} and \eqref{D2u-2}, we obtain
\begin{equation}
\label{D2-renewal-nonlinear}
\begin{aligned}
\bE\left|D^2_{\xi_1,\xi_2}u_{n+1}(t,x)\right|^p
&\le
C_{T,p}|z_1z_2|^p
G_{t-r_2}^p(x-y_2)
H_{r_2-r_1}^{(p)}(y_2-y_1)
\\
&\quad+
C_{T,p}
\int_{r_2}^t\int_{\bR}
H_{t-s}^{(p)}(x-y)
\bE\left|D^2_{\xi_1,\xi_2}u_n(s,y)\right|^p
dy ds
\\
&\quad+
C_{T,p}
\int_{r_2}^t\int_{\bR}
H_{t-s}^{(p)}(x-y)
\min_{i=1,2}
\left\{
|z_i|^pH_{s-r_i}^{(p)}(y-y_i)
\right\}
dy ds .
\end{aligned}
\end{equation}
 For $n\ge2$, iterating \eqref{D2-renewal-nonlinear} and using
$D^2_{\xi_1,\xi_2}u_1=0$, we obtain
\begin{align*}
\bE\left|D^2_{\xi_1,\xi_2}u_n(t,x)\right|^p&\le C_{T,p}|z_1z_2|^p
H_{r_2-r_1}^{(p)}(y_2-y_1) \left[ G_{t-r_2}^p(x-y_2)+ \sum_{k=1}^{n-2}
I_1^{(k)}\right]+\sum_{k=1}^{n-1}
 I_2^{(k)}
\end{align*}
where
\begin{equation*}
I_1^{(k)}= C_{T,p}^k
\int_{T_k(r_2,t)}
\int_{\bR^k}
\prod_{i=1}^{k}
H_{t_{i+1}-t_i}^{(p)}(x_{i+1}-x_i)\cdot
G_{t_1-r_2}^p(x_1-y_2)
d\pmb{x}d\pmb{t}  
\end{equation*}
and 
\begin{equation*}
I_2^{(k)}= C_{T,p}^k
\int_{T_k(r_2,t)}
\int_{\bR^k}
\prod_{i=1}^{k}
H_{t_{i+1}-t_i}^{(p)}(x_{i+1}-x_i) \cdot
\min_{i=1,2}\left\{
|z_i|^pH_{t_1-r_i}^{(p)}(x_1-y_i)
\right\}
d\pmb{x}d\pmb{t},
\end{equation*}
By \eqref{Gp-bound},
\begin{equation}
\label{D2-source-iteration-bound}
G_{t-r_2}^p(x-y_2)
+
\sum_{k=1}^{n-2} I_1^{(k)}
\le
C_{T,p}H_{t-r_2}^{(p)}(x-y_2).
\end{equation}

For each fixed $j\in\{1,2\}$, since $\min_{i=1,2}
\left\{
|z_i|^pH_{t_1-r_i}^{(p)}(x_1-y_i)
\right\}
\le
|z_j|^pH_{t_1-r_j}^{(p)}(x_1-y_j)$, and since $T_k(r_2,t)\subset T_k(r_j,t)$ for $j=1,2$, we obtain
\begin{align*}
\sum_{k=1}^{n-1} I_2^{(k)}
&\le
\sum_{k=1}^{n-1}
C_{T,p}^k
\int_{T_k(r_j,t)}
\int_{\bR^k}
\prod_{i=1}^{k}
H_{t_{i+1}-t_i}^{(p)}(x_{i+1}-x_i)
|z_j|^pH_{t_1-r_j}^{(p)}(x_1-y_j)
d\pmb{x}d\pmb{t}.
\end{align*}
Therefore, by \eqref{GpG2-bound}
\begin{equation*}
\sum_{k=1}^{n-1} I_2^{(k)}
\le
C_{T,p}
|z_j|^pH_{t-r_j}^{(p)}(x-y_j),
\qquad j=1,2.
\end{equation*}
Taking the minimum over $j=1,2$, we get
\begin{equation}
\label{D2-min-iteration-bound}
\sum_{k=1}^{n-1} I_2^{(k)}
\le
C_{T,p}
\min\left\{
|z_1|^pH_{t-r_1}^{(p)}(x-y_1),
|z_2|^pH_{t-r_2}^{(p)}(x-y_2)
\right\}.
\end{equation}

Using \eqref{D2-source-iteration-bound} and
\eqref{D2-min-iteration-bound}, we obtain, uniformly in $n\ge2$,
\begin{align}
\label{D2u-G-nonlinear}
\bE\left|D^2_{\xi_1,\xi_2}u_n(t,x)\right|^p
&\le
C_{T,p}|z_1z_2|^p
H_{t-r_2}^{(p)}(x-y_2)
H_{r_2-r_1}^{(p)}(y_2-y_1)
\nonumber\\
&\quad+
C_{T,p}
\min\left\{
|z_1|^pH_{t-r_1}^{(p)}(x-y_1),
|z_2|^pH_{t-r_2}^{(p)}(x-y_2)
\right\}.
\end{align}

Taking $p=2$ in \eqref{D2u-G-nonlinear} and integrating with respect to
$\fm(d\xi_1)\fm(d\xi_2)$ gives
\[
\sup_{n\ge2}
\bE\|D^2u_n(t,x)\|_{\cH^{\otimes2}}^2<\infty.
\]
The product term is integrable because $m_2<\infty$. For the minimum
term, use

\[
\min\{|z_1|^2A,|z_2|^2B\}\le |z_1z_2|A^{1/2}B^{1/2},
\qquad (H_s^{(2)})^{1/2}=\sqrt2\,G_s.
\]
Its contribution is bounded by
\[
\begin{aligned}
&\int_0^t\int_0^t\int_{\bR^2}\int_{\bR_0^2}
\min\left\{|z_1|^2H_{t-r_1}^{(2)}(x-y_1),
|z_2|^2H_{t-r_2}^{(2)}(x-y_2)\right\}
\nu(dz_1)\nu(dz_2)dy_1dy_2dr_1dr_2 \\
&\qquad\le
2m_1^2
\int_0^t\int_0^t
\left(\int_{\bR}G_{t-r_1}(x-y_1)\,dy_1\right)
\left(\int_{\bR}G_{t-r_2}(x-y_2)\,dy_2\right)
dr_1dr_2  \\
&\qquad=
2m_1^2t^2
<\infty .
\end{aligned}
\]
 Thus the minimum term is integrable because $m_1<\infty$.

\medskip
Apply \cite[Lemma~2.1]{BLW26} with $F_n=u_n(t,x)$ and $F=u(t,x)$
for fixed $(t,x)\in\bR_+\times\bR$. It follows that
$u(t,x)\in{\rm dom}(D^2)$ and that $D^2u_n(t,x)$ converges weakly to
$D^2u(t,x)$ in $L^2(\Omega;\cH^{\otimes2})$. This proves (i). Part (ii)
follows from \cite[Lemma~A.1]{BS26}.
\end{proof}

\section{Gaussian fluctuations of the spatial averages}
\label{sec:Gaussian-fluctuations}

We now prove the Gaussian fluctuation results. Ergodicity, the limiting
covariance, and the FCLT follow the pAm arguments after inserting the nonlinear
moment bounds. The QCLT additionally requires control of the minimum term in
the second Malliavin derivative estimate.

For $0\le r<t$, $R>0$ and $y\in\bR$, set
\begin{equation}
\label{def-phi}
\varphi_{t,R}(r,y)
=
\int_{-R}^{R}G_{t-r}(x-y)\,dx,
\qquad
\varphi_{t,R}(y)=\varphi_{t,R}(0,y).
\end{equation}

\subsection{Ergodicity and limiting covariance}
\begin{theorem}
\label{ergodic-cov}
\begin{itemize}
\item[(i)] For every $t>0$, the random field $\{u(t,x)\}_{x\in\bR}$ is
strictly stationary and ergodic. Consequently,
\[
\lim_{R\to\infty}\frac{1}{R}F_R(t)=0
\qquad
\text{a.s. and in }L^2(\Omega).
\]

\item[(ii)] For every $s,t>0$,
\begin{equation}
\label{limit-cov}
\lim_{R\to\infty}
\frac{1}{R}\bE[F_R(t)F_R(s)]
=
\Sigma_{t,s}
:=
2m_2\int_0^{t\wedge s}
\bE\big|\sigma(u(r,0))\big|^2\,dr.
\end{equation}
In particular,
\[
\sigma_R^2(t):=\operatorname{Var}(F_R(t))
\sim
\Sigma_{t,t}R,
\qquad R\to\infty.
\]
\end{itemize}
\end{theorem}

\begin{proof}
 Part~(i) follows as in \cite[Theorem~1.1]{BZ24}: stationarity follows
from the spatial homogeneity of $L$, and ergodicity from
\cite[Lemma~4.2]{BZ24} and Theorem~\ref{key-th-Du} with $p=2$.

As in \cite[proof of Theorem~1.1(ii)]{BLW26}, stochastic Fubini's
theorem, the It\^o isometry, and stationarity give
\[
\bE[F_R(t)F_R(s)]
=
m_2\int_0^{t\wedge s}
\bE|\sigma(u(r,0))|^2
\int_{\bR}
\varphi_{t,R}(r,y)\varphi_{s,R}(r,y)\,dy\,dr,
\]
 and, for $r<t\wedge s$,
\[
\frac1R\int_{\bR}
\varphi_{t,R}(r,y)\varphi_{s,R}(r,y)\,dy
\longrightarrow 2,
\qquad R\to\infty.
\]
Dominated convergence yields \eqref{limit-cov}.
\end{proof}

\subsection{Quantitative central limit theorem}
We first introduce some notation. For \(t>0\), \(R>0\),
\(p>0\) and \(y\in\bR\), define
\[
K_t^{(p)}(y)
=
G_t(y)+G_{pt}(y),
\qquad
\Phi_{t,R}^{(p)}(y)
=
\int_{-R}^{R}K_t^{(p)}(x-y)\,dx .
\]

We use the estimates from \cite[Lemmas~4.3--4.5]{BLW26}, collected
below for convenience.
\begin{lemma}
Fix $T>0$. For every $0<t\le T$, $R>0$ and $p>0$,
\begin{align}
\label{Phi-bound}
\Phi_{t,R}^{(p)}(y)
&\le 2,
 \\
\label{K-power-bound}
\big(K_t^{(p)}(x)\big)^\theta
&\lesssim K_{t/\theta}^{(p)}(x),
\qquad 0<\theta<1,\\
\label{K-power-bound-large}
\big(K_t^{(p)}(x)\big)^\theta
&\lesssim t^{-\frac{\theta-1}{2}}K_{t/\theta}^{(p)}(x),
\qquad \theta>1.
\end{align}
\end{lemma}

The following result is the nonlinear analogue of
\cite[Lemma~4.2]{BLW26}. It follows by integrating the estimates of
Theorems~\ref{key-th-Du} and \ref{key-th-D2u} over $[-R,R]$, so we omit its proof.

\begin{lemma}
\label{lem:FR-derivative-nonlinear}
Let \(p\in[1,\frac32)\) satisfy \(m_{2p}<\infty\), and assume that $m_1<\infty$.

\begin{itemize}
\item[(i)] For any \(0\le r<t\le T\), \(y\in\bR\) and \(z\in\bR_0\),
\[
\|D_{r,y,z}F_R(t)\|_{2p}
\lesssim
|z|\Phi^{(p)}_{t-r,R}(y).
\]

\item[(ii)] Let \(\xi_i=(r_i,y_i,z_i)\), \(i=1,2\). For any
\(0\le r_1,r_2<t\le T\) with \(r_1\neq r_2\), \(y_1,y_2\in\bR\) and
\(z_1,z_2\in\bR_0\),
\[
\|D^2_{\xi_1,\xi_2}F_R(t)\|_{2p}
\lesssim
A_R(\xi_1,\xi_2)+B_R(\xi_1,\xi_2),
\]
where
\[
A_R(\xi_1,\xi_2)
=
|z_1z_2|
\begin{cases}
K^{(p)}_{r_2-r_1}(y_2-y_1)\Phi^{(p)}_{t-r_2,R}(y_2),
& r_1<r_2,\\[1mm]
K^{(p)}_{r_1-r_2}(y_1-y_2)\Phi^{(p)}_{t-r_1,R}(y_1),
& r_2<r_1,
\end{cases}
\]
and
\[
B_R(\xi_1,\xi_2)
=
\int_{-R}^{R}
\min\Big\{
|z_1|K^{(p)}_{t-r_1}(x-y_1),
|z_2|K^{(p)}_{t-r_2}(x-y_2)
\Big\}\,dx.
\]
Moreover, for \(0<\theta<1\), define
\[
\Psi_{\theta,R}(r_1,y_1,r_2,y_2)
=
\int_{-R}^{R}
\big[K^{(p)}_{t-r_1}(x-y_1)\big]^\theta
\big[K^{(p)}_{t-r_2}(x-y_2)\big]^{1-\theta}
\,dx.
\]
Then the interpolation inequality \(\min\{a,b\}\le a^\theta b^{1-\theta}\)
gives
\begin{equation}
\label{theta}
B_R(\xi_1,\xi_2)
\le
|z_1|^\theta |z_2|^{1-\theta}
\Psi_{\theta,R}(r_1,y_1,r_2,y_2).
\end{equation}
\end{itemize}
\end{lemma}

\begin{proof}[ Proof of Theorem~\ref{QCLT}]
Fix \(t>0\). By Theorem~\ref{ergodic-cov},
\(\sigma_R^2(t)\sim \Sigma_{t,t}R\), and hence
\(\sigma_R^{-2p}(t)\lesssim R^{-p}\). Since \(m_1+m_{2p}<\infty\),
we have \(m_a<\infty\) for all \(a\in[1,2p]\). We apply the second-order Poincar\'e inequality on the Poisson space, in the
form of \cite[Theorem~3.4]{T25}.
Let \(\gamma_1,\ldots,\gamma_7\) be the corresponding quantities. It remains
to prove that
\begin{equation}
\label{gamma-goal}
\gamma_i
\lesssim
R^{-\left(1-\frac1p\right)},\qquad i=1,\ldots,7.
\end{equation}

{\bf Estimate of $\gamma_1$.}
\begin{equation}
\label{gamma1}
\gamma_1^p
\lesssim
R^{-p}
\int_{\bf Z}
\left[
\int_{\bf Z}
\big\|D_{\xi_2}F_R(t)\big\|_{2p}
\big\|D^2_{\xi_1,\xi_2}F_R(t)\big\|_{2p}
\fm(d\xi_2)
\right]^p
\fm(d\xi_1).
\end{equation}
By Lemma~\ref{lem:FR-derivative-nonlinear},
\[
\gamma_1^p
\lesssim
\gamma_{1,A}^p+\gamma_{1,B}^p,
\]
where
\[
\gamma_{1,A}^p
:=
R^{-p}
\int_{\bf Z}
\left[
\int_{\bf Z}
|z_2|\Phi_{t-r_2,R}^{(p)}(y_2) A_R(\xi_1,\xi_2)\,\fm(d\xi_2)
\right]^p
\fm(d\xi_1),
\]
and
\[
\gamma_{1,B}^p
:=
R^{-p}
\int_{\bf Z}
\left[
\int_{\bf Z}
|z_2|\Phi_{t-r_2,R}^{(p)}(y_2)  B_R(\xi_1,\xi_2)\,\fm(d\xi_2)
\right]^p
\fm(d\xi_1).
\]
 The $A_R$ contribution is the pAm term estimated in \cite{BLW26}.
We therefore treat only the new $B_R$ contribution. We use the same
decomposition for $\gamma_2,\gamma_5,\gamma_6$, and $\gamma_7$.

 Applying \eqref{theta} with $\theta=1/p$ gives
\begin{equation}
\label{gamma1-B}
\gamma^p_{1,B}\lesssim R^{-p}m_1m_{2-\frac{1}{p}}^p
\int_0^t\int_{\bR}
\left[
\int_0^t\int_{\bR}
\Phi^{(p)}_{t-r_2,R}(y_2)
\Psi_{1/p,R}(r_1,y_1,r_2,y_2)
\,dy_2dr_2
\right]^p
dy_1dr_1 .
\end{equation}
By \eqref{Phi-bound} and \eqref{K-power-bound},
\[
\int_0^t\int_{\bR}
\Phi_{t-r_2,R}^{(p)}(y_2)
\big[K_{t-r_2}^{(p)}(x-y_2)\big]^{1-1/p}
dy_2dr_2
\lesssim 1.
\]
Hence
\begin{align*}
 \gamma_{1,B}^p
&\lesssim
R^{-p}
\int_0^t\int_{\bR}
\left[
\int_{-R}^{R}
\big[K_{t-r_1}^{(p)}(x-y_1)\big]^{1/p}
dx
\right]^p
dy_1dr_1\\
&= R^{-p}
\int_0^t
\left\|\mathbf 1_{[-R,R]}*\big[K_{t-r_1}^{(p)}\big]^{1/p}\right\|_{L^p(\bR)}^p
\,dr_1  \lesssim R^{-(p-1)}.
\end{align*}
Here the last inequality uses Young's inequality and \eqref{K-power-bound}.

\medskip
{\bf Estimate of $\gamma_2$.}  Applying \eqref{theta} with
$\theta=1/(2p)$ gives
\begin{equation}
\label{gamma2-B}
(\gamma_{2,B})^p \lesssim m_1m_{2-1/p}^p
R^{-p}
\int_0^t\int_{\bR}
\left[
\int_0^t\int_{\bR}
\Psi_{1/(2p),R}(r_1,y_1,r_2,y_2)^2
\,dy_2dr_2
\right]^p
dy_1dr_1 .
\end{equation}
For fixed $r_2$, by Young's convolution inequality,
\begin{align*}
&\int_{\bR}
\Psi_{1/(2p),R}(r_1,y_1,r_2,y_2)^2
dy_2 \le
\left(
\int_{-R}^{R}
[K^{(p)}_{t-r_1}(x-y_1)]^{\frac{1}{2p}} dx
\right)^2
\int_{\bR}
[K^{(p)}_{t-r_2}(y)]^{2(1-\frac{1}{2p})}dy .
\end{align*}
By \eqref{K-power-bound-large},
\begin{equation*}
    \int_0^t\int_{\bR}
\Psi_{\theta,R}(r_1,y_1,r_2,y_2)^2\,dy_2dr_2
\lesssim
\left[
\int_{-R}^{R}
\big[K_{t-r_1}^{(p)}(x-y_1)\big]^{\frac{1}{2p}} dx
\right]^2 \int_0^t
(t-r_2)^{-\frac12(1-\frac1p)}dr_2.
\end{equation*}
The second integral is finite and hence 
\begin{align*}
(\gamma_{2,B})^p
&\lesssim
R^{-p}
\int_0^t\int_{\bR}
\left[
\int_{-R}^{R}
\big[K_{t-r_1}^{(p)}(x-y_1)\big]^{\frac{1}{2p}} dx
\right]^{2p}
dy_1dr_1\\
& =
R^{-p}
\int_0^t
\left\|
\mathbf 1_{[-R,R]}*
\big[K_{t-r_1}^{(p)}\big]^{\frac1{2p}}
\right\|_{L^{2p}(\bR)}^{2p}
\,dr_1   \lesssim R^{-(p-1)}.
\end{align*}

{\bf Estimates of $\gamma_3$ and $\gamma_4$.}
The quantities $\gamma_3$ and $\gamma_4$ involve only the first Malliavin
derivative. Hence their estimates are the same as in the pAm case. 

\medskip
{\bf Estimate of $\gamma_5$.}  Applying \eqref{theta} with
$\theta=1/(2p)$ gives
\begin{equation}
\label{gamma5}
\gamma_{5,B}^p
\lesssim
m_1m_{2p-1}R^{-p}
\int_0^t\int_0^t\int_{\bR^2}
\Psi_{1/(2p),R}(r_1,y_1,r_2,y_2)^{2p}
\,dy_1dy_2dr_1dr_2 .
\end{equation}
 By Young's inequality,
\begin{align*}
\int_{\bR}
\Psi_{1/(2p),R}(r_1,y_1,r_2,y_2)^{2p}\,dy_1 &=\left\|
\mathbf 1_{[-R,R]}(\cdot)
\big[K_{t-r_2}^{(p)}(\cdot-y_2)\big]^{1-\frac1{2p}}
*
\big[K_{t-r_1}^{(p)}\big]^{\frac1{2p}}
\right\|_{L^{2p}(\bR)}^{2p} \\
&\le  \left\|
\mathbf 1_{[-R,R]}(\cdot)
\big[K_{t-r_2}^{(p)}(\cdot-y_2)\big]^{1-\frac1{2p}}
\right\|_{L^{2p}(\bR)}^{2p}
\left\|
\big[K_{t-r_1}^{(p)}\big]^{\frac1{2p}}
\right\|_{L^1(\bR)}^{2p}\\
&\lesssim \int_{-R}^{R}
\big[K_{t-r_2}^{(p)}(x-y_2)\big]^{2p-1}\,dx .
\end{align*}
Therefore
\begin{align*}
\gamma_{5,B}^p &\lesssim R^{-p} \int_0^t\int_0^t\int_{\bR}\int_{-R}^{R}
\big[K_{t-r_2}^{(p)}(x-y_2)\big]^{2p-1}
\,dxdy_2dr_1dr_2 \\
&\lesssim R^{-p} R \int_0^t\int_0^t
(t-r_2)^{-(p-1)}
\,dr_1dr_2  \lesssim R^{-(p-1)}.
\end{align*}

{\bf Estimate of $\gamma_6$.} Applying \eqref{theta} with
$\theta=1-1/p$ gives
\begin{equation}
\label{gamma6-B}
\gamma_{6,B}^p
\lesssim
m_{2p-1}m_1 R^{-p}
\int_0^t\int_0^t\int_{\bR^2}
\big[\Phi_{t-r_1,R}^{(p)}(y_1)\big]^p
\Psi_{1-\frac1p,R}(r_1,y_1,r_2,y_2)^p
dy_1dy_2dr_1dr_2.
\end{equation}
For fixed \(r_1,r_2,y_1\), by Young's inequality,
\begin{align*}
\int_{\bR}
\Psi_{1-\frac1p,R}(r_1,y_1,r_2,y_2)^p dy_2 &=   \left\|
\mathbf 1_{[-R,R]}(\cdot)
\big[K_{t-r_1}^{(p)}(\cdot-y_1)\big]^{1-\frac1p}
*
\big[K_{t-r_2}^{(p)}\big]^{\frac1p}
\right\|_{L^p(\bR)}^p  \\
&\le \left\|
\mathbf 1_{[-R,R]}(\cdot)
\big[K_{t-r_1}^{(p)}(\cdot-y_1)\big]^{1-\frac1p}
\right\|_{L^p(\bR)}^p
\left\|
\big[K_{t-r_2}^{(p)}\big]^{\frac1p}
\right\|_{L^1(\bR)}^p  \\
&\lesssim \int_{-R}^{R}
\big[K_{t-r_1}^{(p)}(x-y_1)\big]^{p-1}\,dx
\lesssim 1,
\end{align*}
where the last inequality follows from \eqref{K-power-bound}, since
\(p-1\in(0,1)\). Therefore, from \eqref{gamma6-B}
\begin{align*}
\gamma_{6,B}^p
&\lesssim  R^{-p}
\int_0^t\int_0^t\int_{\bR}
\big[\Phi_{t-r_1,R}^{(p)}(y_1)\big]^p
\,dy_1dr_1dr_2  \lesssim R^{-(p-1)}.
\end{align*}

{\bf Estimate of $\gamma_7$.}  Applying \eqref{theta} with
$\theta=2p-2$ gives
\begin{align}
\label{gamma7}
\gamma_{7,B}^p
&\lesssim m_{2p-1}m_1 R^{-p}
\int_0^t\int_0^t\int_{\bR^2}
\Phi_{t-r_1,R}^{(p)}(y_1)
\nonumber\\[-1mm]
&\quad\times
\big[\Phi_{t-r_2,R}^{(p)}(y_2)\big]^{2p-2}
\Psi_{\theta,R}(r_1,y_1,r_2,y_2)
dy_1dy_2dr_1dr_2.
\end{align}
Since $2p-2,3-2p\in(0,1)$, \eqref{Phi-bound} and
\eqref{K-power-bound} imply, uniformly in $x,r_1,r_2$,
\begin{align*}
\int_{\bR}
\Phi_{t-r_1,R}^{(p)}(y_1)
\big[K_{t-r_1}^{(p)}(x-y_1)\big]^{2p-2}dy_1
&\lesssim1,\\
\int_{\bR}
\big[\Phi_{t-r_2,R}^{(p)}(y_2)\big]^{2p-2}
\big[K_{t-r_2}^{(p)}(x-y_2)\big]^{3-2p}dy_2
&\lesssim1.
\end{align*}
 Fubini's theorem in \eqref{gamma7} therefore gives
$\gamma_{7,B}^p\lesssim R^{-p}\int_{-R}^{R}\,dx\lesssim R^{-(p-1)}$.

\medskip
Combining the above estimates, we obtain \eqref{gamma-goal}.
\end{proof}

\subsection{Functional central limit theorem}

\begin{proof}[ Proof of Theorem~\ref{FCLT}]
We verify convergence of the finite-dimensional distributions and
tightness. The finite-dimensional distributions converge by the argument in
\cite[proof of Theorem~1.3]{BLW26}, with the covariance limit from
Theorem~\ref{ergodic-cov}. For tightness, stochastic Fubini's theorem gives
\begin{equation}
\label{FR-decomp}
F_R(t)
=
\int_0^t\int_{\bR}
\varphi_{t,R}(s,y)\sigma(u(s,y))L(ds,dy)
=
C_R(t)+M_R(t),
\end{equation}
where
\begin{align*}
C_R(t)
&=
\int_0^t\int_{\bR}
\left(
\varphi_{t,R}(s,y)-\mathbf 1_{[-R,R]}(y)
\right)
\sigma(u(s,y))L(ds,dy), \\
M_R(t)
&=
\int_0^t\int_{\bR}
\mathbf 1_{[-R,R]}(y)
\sigma(u(s,y))L(ds,dy).
\end{align*}
Since \(\sigma\) is Lipschitz and \(2p<3\), Theorem~\ref{thm-u-pth-moment-nonlinear}
applied with exponent \(2p\) yields
\[
\sup_{(t,x)\in[0,T]\times\bR}
\bE|\sigma(u(t,x))|^{2p}<\infty .
\]

 This uniform moment bound is precisely the input needed in the
tightness estimates of \cite[proof of Theorem~1.3]{BLW26}. Repeating those
estimates for $C_R$ and $M_R$ proves tightness in both the $J_1$ and uniform
topologies, and completes the proof.
\end{proof}

\medskip

{\it Acknowledgment.}
 The authors thank Raluca Balan for helpful comments.


\begin{thebibliography}{99}

\bibitem{BLW26}
Balan, R.M., Le Gall, M. and Wang, J. (2026).
Gaussian fluctuations for the parabolic Anderson model with L\'evy white noise.
Preprint, arXiv:2607.02742.

\bibitem{BN17}
Balan, R.M. and Ndongo, C.B. (2017).
Malliavin differentiability of solutions of SPDEs with L\'evy white noise.
{\em Int. J. Stoch. Anal.}
{\bf 2017}, Article ID 9693153, 1--9.

\bibitem{BS26}
Balan, R.M. and Salins, M. (2026).
Gaussian fluctuations for the nonlinear stochastic heat equation with drift.
{\em Electron. J. Probab.}
{\bf 31}, Paper No.~122, 1--31.

\bibitem{BZ24}
Balan, R.M. and Zheng, G. (2024).
Hyperbolic Anderson model with L\'evy white noise:
spatial ergodicity and fluctuation.
{\em Trans. Amer. Math. Soc.}
{\bf 377}, 4171--4221.

\bibitem{BZ26}
Balan, R.M. and Zheng, G. (2026).
Central limit theorem for stochastic nonlinear wave equation
with pure-jump L\'evy white noise.
{\em Electron. J. Probab.}
{\bf 31}, Paper No.~70, 1--46.

\bibitem{C17}
Chong, C. (2017).
Stochastic PDEs with heavy-tailed noise.
{\em Stochastic Process. Appl.}
{\bf 127}, 2262--2280.

\bibitem{CDH19}
Chong, C., Dalang, R.C. and Humeau, T. (2019).
Path properties of the solution to the stochastic heat equation with L\'evy
noise.
{\em Stoch. Partial Differ. Equ. Anal. Comput.}
{\bf 7}, 123--168.

\bibitem{CK19}
Chong, C. and Kevei, P. (2019).
Intermittency for the stochastic heat equation with L\'evy noise.
{\em Ann. Probab.}
{\bf 47}, 1911--1948.

\bibitem{HNV20}
Huang, J., Nualart, D. and Viitasaari, L. (2020).
A central limit theorem for the stochastic heat equation.
{\em Stochastic Process. Appl.}
{\bf 130}, 7170--7184.

\bibitem{HNVZ20}
Huang, J., Nualart, D., Viitasaari, L. and Zheng, G. (2020).
Gaussian fluctuations for the stochastic heat equation with colored noise.
{\em Stoch. Partial Differ. Equ. Anal. Comput.}
{\bf 8}, 402--421.

\bibitem{KNP21}
Khoshnevisan, D., Nualart, D. and Pu, F. (2021).
Spatial stationarity, ergodicity, and CLT for parabolic Anderson model
with delta initial condition in dimension $d\geq 1$.
{\em SIAM J. Math. Anal.}
{\bf 53}, 2084--2133.

\bibitem{NN18}
Nualart, D. and Nualart, E. (2018).
{\em Introduction to Malliavin Calculus}.
Cambridge University Press, Cambridge.

\bibitem{NS24}
Nualart, D. and Saikia, B. (2024).
Gaussian fluctuations of spatial averages of a system of stochastic heat
equations.
{\em Stoch. Anal. Appl.}
{\bf 42}, 1185--1196.

\bibitem{NXZ22}
Nualart, D., Xia, P. and Zheng, G. (2022).
Quantitative central limit theorems for the parabolic Anderson model driven
by colored noises.
{\em Electron. J. Probab.}
{\bf 27}, Paper No.~120, 1--43.

\bibitem{T25}
Trauthwein, T. (2025).
Quantitative CLTs on the Poisson space via Skorohod estimates and
$p$-Poincar\'e inequalities.
{\em Ann. Appl. Probab.}
{\bf 35}, 1716--1754.

\end{thebibliography}
\end{document}